\documentclass[11pt]{amsart}

\usepackage{amsmath,amssymb,amsthm}
\usepackage{xcolor}
\usepackage{hyperref}

\numberwithin{equation}{section}

\newtheorem{theorem}{Theorem}[section]
\newtheorem{lemma}[theorem]{Lemma}
\newtheorem{corollary}[theorem]{Corollary}

\theoremstyle{definition}
\newtheorem{definition}[theorem]{Definition}

\title[TV distance for Bernoulli products]{TV over Bernoulli products: the small parameter regime}

\author{Ariel Avital}
\email{avitalq@post.bgu.ac.il}

\author{Aryeh Kontorovich}
\email{karyeh@cs.bgu.ac.il}

\author{George Salafatinos}
\email{georgesalafatinos@gmail.com}

\date{} %

\newcommand{\beq}{\begin{eqnarray*}}
\newcommand{\eeq}{\end{eqnarray*}}
\newcommand{\beqn}{\begin{eqnarray}}
\newcommand{\eeqn}{\end{eqnarray}}

\newcommand{\TV}{\operatorname{TV}}

\newcommand{\Ber}{\operatorname{Ber}}

\newcommand{\half}{\frac{1}{2}}

\newcommand{\abs}[1]{\left|#1\right|}

\newcommand{\norm}[1]{\Vert #1 \Vert}
\newcommand{\set}[1]{\left\{ #1 \right\}}

\newcommand{\defref}[1]{Definition~\ref{#1}}

\newcommand{\vct}[1]{\mathbf{#1}}
\newcommand{\onevec}{\mathbf{1}}
\renewcommand{\vec}[1]{\textup{{\bf #1}}}

\newcommand{\pvec}{\vct{p}}
\newcommand{\qvec}{\vct{q}}
\newcommand{\xvec}{\vct{x}}
\newcommand{\yvec}{\vct{y}}
\newcommand{\uvec}{\vct{u}}

\newcommand{\normone}[1]{\left\lVert #1 \right\rVert_1}

\newcommand{\lam}{\lambda}

\begin{document}

\begin{abstract}
We study the total variation distance (TV) between two
$n$-fold Bernoulli product measures
parametrized by
$\vec p=(p_1,\ldots,p_n)$
and
$\vec q=(q_1,\ldots,q_n)$,
respectively, in the \emph{tiny} and \emph{small} regimes.
In the tiny regime, we have $p_i,q_i\lesssim 1/n^2$,
and in the small regime, $p_i,q_i\lesssim 1/n$.
We discover that in the tiny regime,
the TV distance behaves as $\|\vec p-\vec q\|_1$,
while in the small regime, it behaves as
\[
\sum_{i=1}^n
\Big|
p_i\prod_{j\neq i}(1-p_j)
-
q_i\prod_{j\neq i}(1-q_j)
\Big|,
\]
both up to absolute constants.
Along the way we discover some identities of possible independent
interest.
\end{abstract}

\maketitle

\section{Introduction}

For $\pvec,\qvec\in[0,1]^n$, consider the Bernoulli product measures
\[
\Ber(\pvec):=\Ber(p_1)\otimes\cdots\otimes\Ber(p_n),
\qquad
\Ber(\qvec):=\Ber(q_1)\otimes\cdots\otimes\Ber(q_n),
\]
on the Hamming cube $\{0,1\}^n$.
The total variation distance
\[
\TV(\Ber(\pvec),\Ber(\qvec))
=\half\sum_{x\in\{0,1\}^n}\abs{\Ber(\pvec)(x)-\Ber(\qvec)(x)}
\]
is both fundamental and notoriously difficult to compute exactly.
From the algorithmic perspective,
computing $\TV(\Ber(\pvec),\Ber(\qvec))$ exactly
is $\#\mathrm{P}$-hard in general \cite{DBLP:conf/ijcai/0001GMMPV23},
motivating
a series of
efficient approximation schemes for product measures
\cite{FengApproxTV23,Feng24Deterministically,bhattacharyya2024total}.
From the analytic perspective,
TV is often bounded
by more tractable divergences (KL, Hellinger, $\chi^2$, etc.),
whose tensorization properties are well understood;
see \cite{kon25tens} for a recent discussion in the present setting.

This work continues the program initiated in \cite{kon25tens,kontorovich2026tvhomogenization} of
approximating the TV over product measures in terms of readily computable elementary functions; the forthcoming paper \cite{KonAv26} further builds on the results we prove here. 
Our point of departure is 
Theorems 1.1 and 1.2 of
\cite{kon25tens}, which show, respectively, that
\[
\TV(\Ber(\pvec),\Ber(\qvec))\ \gtrsim\ \norm{\pvec-\qvec}_2,
\]
and
\[
\TV(\Ber(\pvec),\Ber(\qvec))\ \le\ \norm{\pvec-\qvec}_2,
\]
the latter
under the additional assumption that $\pvec,\qvec$
are
symmetric about $1/2$
(i.e., $\pvec=\onevec-\qvec$).
In the Appendix, we extend the argument to \emph{quasi-symmetric} pairs, obtaining the upper bound with an extra factor $\sqrt{2}$.
Taken together, these results indicate that in a neighborhood of $1/2$, TV behaves like $\ell_2$, up to universal constants.

The focus of the present note is complementary: we turn to the opposite extreme of the parameter space, namely $p_i,q_i$ very close to $0$ (and, by symmetry, very close to $1$), and identify regimes in which $\TV\big(\Ber(\pvec),\Ber(\qvec)\big)$ admits a particularly simple description.
A key tool is the \emph{slice decomposition}: if $\Delta_k(\pvec,\qvec)$ denotes the total absolute discrepancy between the two measures on the $k$-th Hamming slice, then
\beqn
\label{eq:tv-delta}
2\TV(\Ber(\pvec),\Ber(\qvec))=\sum_{k=0}^n \Delta_k(\pvec,\qvec),
\eeqn
so controlling TV reduces to understanding which slices dominate.

\medskip\noindent\textbf{Main results.}
For a subset $S\subseteq[n]$ and a parameter vector $\yvec\in[0,1]^n$,
let
$P_S(\yvec)$ denote the mass assigned by $\Ber(\yvec)$ to the atom with ones in~$S$
(see \eqref{eq:PS_def} below),
and define the slice discrepancies $\Delta_k(\pvec,\qvec)$ as in \defref{def:slice_discrepancies}.
Our first theorem shows that in the tiny regime, the TV distance is equivalent to $\normone{\pvec-\qvec}$ up to constants.

\begin{theorem}[Tiny regime: $\ell_1$ geometry]\label{thm:intro_tiny_tv}
For $\pvec,\qvec\in[0,1/n^2]^n$,
\[
\frac14\,\normone{\pvec-\qvec}
\ \le\
\TV(\Ber(\pvec),\Ber(\qvec))
\ \le\
\normone{\pvec-\qvec}.
\]
\end{theorem}

Our second theorem shows that in the small regime, the entire $\ell_1$ distance between the product measures is controlled by the singleton slice.

\begin{theorem}[Small regime: singletons control TV]\label{thm:intro_small_tv}
For $\pvec,\qvec\in[0,1/(2n)]^n$
with $n\ge2$, write
\[
\Delta_1(\pvec,\qvec)
=
\sum_{i=1}^n
\Big|
p_i\prod_{j\neq i}(1-p_j)
-
q_i\prod_{j\neq i}(1-q_j)
\Big|.
\]
Then
\[
\frac12\,\Delta_1(\pvec,\qvec)
\ \le\
\TV(\Ber(\pvec),\Ber(\qvec))
\ \le\
\left(2-\frac1n\right)\Delta_1(\pvec,\qvec).
\]
\end{theorem}

\section{Definitions and notation}\label{sec:setup}
We write $A\lesssim B$ to mean $A\le C B$ for a universal constant $C>0$,
and $A\gtrsim B$ similarly; $A\asymp B$ means both $A\lesssim B$ and $B\lesssim A$.
For integer $n\ge 1$, $[n]:=\{1,2,\dots,n\}$.
For a parameter vector $\yvec=(y_1,\dots,y_n)\in[0,1]^n$ and a subset $S\subseteq[n]$, define the product Bernoulli mass
\begin{equation}\label{eq:PS_def}
P_S(\yvec)
\;:=\;
\prod_{i\in S} y_i\,\prod_{i\notin S}(1-y_i).
\end{equation}
In particular, $P_\emptyset(\yvec)=\prod_{i=1}^n(1-y_i)$ is the probability of the all-zeros atom.

\begin{definition}[Slice discrepancies]\label{def:slice_discrepancies}
For $\pvec,\qvec\in[0,1]^n$ and $S\subseteq[n]$, define
\[
\delta_S(\pvec,\qvec):=P_S(\pvec)-P_S(\qvec).
\]
For each $k\in\{0,1,\dots,n\}$ define the absolute $k$-slice discrepancy
\[
\Delta_k(\pvec,\qvec):=\sum_{S\subseteq[n]:\,|S|=k}\abs{\delta_S(\pvec,\qvec)}.
\]
\end{definition}

We will often abbreviate $\delta_S:=\delta_S(\pvec,\qvec)$ and $\Delta_k:=\Delta_k(\pvec,\qvec)$ when the pair $(\pvec,\qvec)$ is clear. In particular \eqref{eq:tv-delta} holds.
We also write
\[
\xvec:=\pvec-\qvec,\qquad x_i:=p_i-q_i,\qquad \normone{\pvec-\qvec}=\sum_{i=1}^n\abs{x_i}.
\]

\section{Proof of Theorem \ref{thm:intro_tiny_tv}}
\label{sec:tiny}
In this section we assume
\begin{equation}\label{eq:tiny_regime_assumption}
\pvec,\qvec\in\Big[0,\frac{1}{n^2}\Big]^n.
\end{equation}
Only the lower bound 
$
\frac14\normone{\pvec-\qvec}
\le
\TV(\Ber(\pvec),\Ber(\qvec))
$
requires proof;
the upper bound
$
\TV(\Ber(\pvec),\Ber(\qvec))
\le
\normone{\pvec-\qvec}
$ is classic \cite[Disp. (1.4)]{kon25tens}.

For each $i\in[n]$, define
\[
P_{-i}(\yvec):=\prod_{j\neq i}(1-y_j),\qquad \text{so that}\qquad P_{\{i\}}(\yvec)=y_i\,P_{-i}(\yvec).
\]

\begin{lemma}
\label{lem:tiny_product_lb}
Under \eqref{eq:tiny_regime_assumption}, for every $i\in[n]$,
\[
P_{-i}(\pvec)=\prod_{j\neq i}(1-p_j)\ \ge\ \frac34.
\]
\end{lemma}

\begin{proof}
If $n=1$, the product is empty and equals $1$.
Assume $n\ge 2$.
Since $p_j\le 1/n^2$, we have
\[
P_{-i}(\pvec)\ge \Big(1-\frac1{n^2}\Big)^{n-1}.
\]
For $n=2$ this equals $3/4$.
For $n\ge 3$, Bernoulli's inequality gives
\[
\Big(1-\frac1{n^2}\Big)^{n-1}\ge 1-\frac{n-1}{n^2}=\frac{n^2-n+1}{n^2}\ge \frac34,
\]
since $4(n^2-n+1)\ge 3n^2$ is equivalent to $(n-2)^2\ge 0$.
\end{proof}

\begin{lemma}
\label{lem:tiny_product_lipschitz}
Under \eqref{eq:tiny_regime_assumption}, for every $i\in[n]$,
\[
\abs{P_{-i}(\pvec)-P_{-i}(\qvec)}\ \le\ \sum_{k\neq i}\abs{x_k}.
\]
\end{lemma}

\begin{proof}
If $n=1$ then $P_{-1}(\cdot)\equiv 1$ and the claim is trivial.
Assume $n\ge 2$ and consider the segment $\uvec(t):=\qvec+t(\pvec-\qvec)$.
The function $t\mapsto P_{-i}(\uvec(t)) $ is differentiable, and by the mean value theorem
\[
P_{-i}(\pvec)-P_{-i}(\qvec)=\sum_{k\neq i}\Big(-\prod_{j\neq i,k}(1-u_j(\xi))\Big)x_k\quad\text{for some }\xi\in(0,1).
\]
Taking absolute values and using $0\le (1-u_j(\xi))\le 1$ gives
\[
\abs{P_{-i}(\pvec)-P_{-i}(\qvec)}\le \sum_{k\neq i}\abs{x_k}.
\]
\end{proof}

\begin{proof}[Proof of Theorem \ref{thm:intro_tiny_tv}]
Fix $i\in[n]$.
Using $P_{\{i\}}(\yvec)=y_iP_{-i}(\yvec)$ we expand
\begin{align*}
P_{\{i\}}(\pvec)-P_{\{i\}}(\qvec)
&=p_iP_{-i}(\pvec)-q_iP_{-i}(\qvec)\\
&=(p_i-q_i)P_{-i}(\pvec)+q_i\big(P_{-i}(\pvec)-P_{-i}(\qvec)\big)\\
&=x_iP_{-i}(\pvec)+q_i\big(P_{-i}(\pvec)-P_{-i}(\qvec)\big).
\end{align*}
By the reverse triangle inequality,
\[
\abs{P_{\{i\}}(\pvec)-P_{\{i\}}(\qvec)}
\ge \abs{x_i}\,P_{-i}(\pvec)-q_i\,\abs{P_{-i}(\pvec)-P_{-i}(\qvec)}.
\]
Apply Lemma~\ref{lem:tiny_product_lb} and Lemma~\ref{lem:tiny_product_lipschitz} to obtain
\[
\abs{P_{\{i\}}(\pvec)-P_{\{i\}}(\qvec)}
\ge \frac34\abs{x_i}-q_i\sum_{k\neq i}\abs{x_k}.
\]
Summing over $i$ yields
\[
\Delta_1\ge \frac34\sum_{i=1}^n\abs{x_i}-\sum_{i=1}^n q_i\sum_{k\neq i}\abs{x_k}
=\frac34\normone{\pvec-\qvec}-\sum_{k=1}^n\abs{x_k}\sum_{i\neq k}q_i.
\]
Because $q_i\le 1/n^2$, for each fixed $k$ we have $\sum_{i\neq k}q_i\le (n-1)/n^2\le 1/4$.
Therefore the last term is at most $\tfrac14\normone{\pvec-\qvec}$, and
\[
\Delta_1\ge \Big(\frac34-\frac14\Big)\normone{\pvec-\qvec}=\frac12\normone{\pvec-\qvec}.
\]
Finally,
\[
\TV(\Ber(\pvec),\Ber(\qvec))
=\frac12\sum_{k=0}^n\Delta_k(\pvec,\qvec)
\ge \frac12\,\Delta_1(\pvec,\qvec)
\ge \frac14\,\normone{\pvec-\qvec}.
\]
The upper bound $\TV(\Ber(\pvec),\Ber(\qvec))\le \normone{\pvec-\qvec}$ is standard and holds for all $\pvec,\qvec\in[0,1]^n$.
\end{proof}

\section{Proof of Theorem \ref{thm:intro_small_tv}}
\label{sec:small}

In this section we assume $n\ge 2$ and work in the 
``small'' parameter domain
\begin{equation}\label{eq:small_regime_domain}
\pvec,\qvec\in[0,\lam_n]^n,\qquad \lam_n:=\frac{1}{2n}.
\end{equation}
It is convenient to also introduce
\begin{equation}\label{eq:beta_def}
\beta_n:=\frac{\lam_n}{1-\lam_n}=\frac{1}{2n-1}.
\end{equation}
All slice discrepancies $\Delta_k$ below are understood for the same pair $(\pvec,\qvec)$.

\begin{theorem}[$\Delta_0$ bound]\label{thm:Delta0_bound}
For $\pvec,\qvec\in[0,\lam_n]^n$,
\[
\Delta_0\ \le\ \left(2-\frac1n\right)\Delta_1
=\frac{2n-1}{n}\,\Delta_1.
\]
\end{theorem}

\begin{proof}
Let $g(\yvec):=P_\emptyset(\yvec)=\prod_{j=1}^n(1-y_j)$ and define the segment $\uvec(t):=\qvec+t(\pvec-\qvec)$.
By the fundamental theorem of calculus,
\[
P_\emptyset(\pvec)-P_\emptyset(\qvec)=\int_0^1\frac{d}{dt}\,g(\uvec(t))\,dt.
\]
Write $x_k:=p_k-q_k$.
Since $\partial g/\partial y_k(\yvec)=-\prod_{j\neq k}(1-y_j)$,
\[
\frac{d}{dt}\,g(\uvec(t))=-\sum_{k=1}^n x_k\prod_{j\neq k}(1-u_j(t)).
\]
Define
\[
B_k:=x_k\int_0^1\prod_{j\neq k}(1-u_j(t))\,dt.
\]
Then $P_\emptyset(\pvec)-P_\emptyset(\qvec)=-\sum_{k=1}^n B_k$, so
$\Delta_0=\abs{P_\emptyset(\pvec)-P_\emptyset(\qvec)}\le \sum_k\abs{B_k}$.

Next fix $k\in[n]$ and consider $h_k(\yvec):=P_{\{k\}}(\yvec)=y_k\prod_{j\neq k}(1-y_j)$.
Differentiating along the same segment gives
\[
\delta_{\{k\}}=h_k(\pvec)-h_k(\qvec)=\int_0^1\frac{d}{dt}\,h_k(\uvec(t))\,dt.
\]
A direct computation yields
\[
\frac{d}{dt}\,h_k(\uvec(t))
=x_k\prod_{j\neq k}(1-u_j(t))\;-
\sum_{m\neq k} x_m\,u_k(t)\prod_{j\neq k,m}(1-u_j(t)).
\]
Integrating and rearranging shows
\[
B_k
=\delta_{\{k\}}+\sum_{m\neq k}x_m\int_0^1 u_k(t)\prod_{j\neq k,m}(1-u_j(t))\,dt.
\]
Taking absolute values and summing over $k$ gives
\begin{align*}
\sum_{k=1}^n\abs{B_k}
&\le \sum_{k=1}^n\abs{\delta_{\{k\}}}
 +\sum_{k=1}^n\sum_{m\neq k}\abs{x_m}\int_0^1 u_k(t)\prod_{j\neq k,m}(1-u_j(t))\,dt.
\end{align*}
For $k\neq m$ we use
\[
u_k(t)\prod_{j\neq k,m}(1-u_j(t))
=\frac{u_k(t)}{1-u_k(t)}\prod_{j\neq m}(1-u_j(t))
\le \beta_n\prod_{j\neq m}(1-u_j(t)),
\]
because $u_k(t)\in[0,\lam_n]$ implies $\frac{u_k(t)}{1-u_k(t)}\le \beta_n$.
Therefore
\[
\abs{x_m}\int_0^1 u_k(t)\prod_{j\neq k,m}(1-u_j(t))\,dt
\le \beta_n\abs{x_m}\int_0^1\prod_{j\neq m}(1-u_j(t))\,dt
=\beta_n\abs{B_m}.
\]
Summing over $k\neq m$ yields
\[
\sum_{k=1}^n\sum_{m\neq k}\abs{x_m}\int_0^1 u_k(t)\prod_{j\neq k,m}(1-u_j(t))\,dt
\le \beta_n(n-1)\sum_{m=1}^n\abs{B_m}.
\]
Hence
\[
\sum_{k=1}^n\abs{B_k}\le \Delta_1+\beta_n(n-1)\sum_{k=1}^n\abs{B_k}.
\]
Since $\beta_n(n-1)=\frac{n-1}{2n-1}<1$, we can absorb to obtain
\[
\sum_{k=1}^n\abs{B_k}\le \frac{1}{1-\beta_n(n-1)}\,\Delta_1=\frac{2n-1}{n}\,\Delta_1.
\]
Finally $\Delta_0\le\sum_k\abs{B_k}$ gives the claim.
\end{proof}

\begin{theorem}[$\Delta_2/\Delta_1$ bound]\label{thm:Delta2_bound}
For $\pvec,\qvec\in[0,\lam_n]^n$,
\[
\Delta_2\ \le\ \frac{3(n-1)}{2(2n-1)}\,\Delta_1.
\]
\end{theorem}

\begin{proof}
Fix $a<b$.
Write $\delta_a:=\delta_{\{a\}}$, $\delta_b:=\delta_{\{b\}}$, and $\delta_{ab}:=\delta_{\{a,b\}}$.

We introduce an auxiliary quantity chosen so that a certain linear combination of singleton and doubleton masses isolates $\delta_{ab}$ cleanly in odds coordinates:
\[
S(\yvec;a,b)
:=\beta_n\big(P_{\{a\}}(\yvec)+P_{\{b\}}(\yvec)\big)-2P_{\{a,b\}}(\yvec).
\]
A direct expansion shows
\begin{equation}\label{eq:deltaab_identity}
\frac{\beta_n}{2}(\delta_a+\delta_b)-\delta_{ab}
=\frac12\Big(S(\pvec;a,b)-S(\qvec;a,b)\Big)
\;=:\;\frac12\,\Delta S_{ab}.
\end{equation}
By the triangle inequality,
\begin{equation}\label{eq:deltaab_bound_step1}
\abs{\delta_{ab}}
\le \frac{\beta_n}{2}\big(\abs{\delta_a}+\abs{\delta_b}\big)+\frac12\abs{\Delta S_{ab}}.
\end{equation}
Summing \eqref{eq:deltaab_bound_step1} over all $a<b$ yields
\begin{equation}\label{eq:Delta2_split}
\Delta_2
\le \frac{\beta_n}{2}(n-1)\Delta_1+\sum_{a<b}\frac12\abs{\Delta S_{ab}}.
\end{equation}

It remains to bound $\sum_{a<b}\abs{\Delta S_{ab}}$.
Introduce odds coordinates $o_i(\yvec):=\frac{y_i}{1-y_i}$.
On $[0,\lam_n]$ we have $0\le o_i(\yvec)\le \beta_n$.
Note the factorizations
\[
P_{\{a\}}(\yvec)=P_\emptyset(\yvec)\,o_a(\yvec),\qquad
P_{\{a,b\}}(\yvec)=P_\emptyset(\yvec)\,o_a(\yvec)o_b(\yvec).
\]
Therefore $S(\yvec;a,b)=P_\emptyset(\yvec)\,H(\yvec;a,b)$ where
\[
H(\yvec;a,b):=\beta_n\big(o_a(\yvec)+o_b(\yvec)\big)-2o_a(\yvec)o_b(\yvec).
\]
Write
$\Delta P_\emptyset:=P_\emptyset(\pvec)-P_\emptyset(\qvec)$
and $\Delta o_i:=o_i(\pvec)-o_i(\qvec)$.
Then
\[
\Delta S_{ab}=P_\emptyset(\pvec)\big(H(\pvec;a,b)-H(\qvec;a,b)\big)+H(\qvec;a,b)\,\Delta P_\emptyset.
\]
Using the bilinear identity
\[
H(\pvec;a,b)-H(\qvec;a,b)
=(\beta_n-2o_b(\qvec))\,\Delta o_a+(\beta_n-2o_a(\pvec))\,\Delta o_b,
\]
and the relations
\[
P_\emptyset(\pvec)\,\Delta o_a=\delta_a-o_a(\qvec)\,\Delta P_\emptyset,
\qquad
P_\emptyset(\pvec)\,\Delta o_b=\delta_b-o_b(\qvec)\,\Delta P_\emptyset,
\]
one obtains the identity
\[
\Delta S_{ab}
=(\beta_n-2o_b(\qvec))\delta_a+(\beta_n-2o_a(\pvec))\delta_b+2o_a(\pvec)o_b(\qvec)\,\Delta P_\emptyset.
\]
Taking absolute values and using $o_i(\cdot)\in[0,\beta_n]$ gives
\[
\abs{\Delta S_{ab}}\le \beta_n\abs{\delta_a}+\beta_n\abs{\delta_b}+2\beta_n^2\abs{\Delta P_\emptyset}.
\]
Summing over $a<b$ and using $\abs{\Delta P_\emptyset}=\Delta_0$ yields
\[
\sum_{a<b}\frac12\abs{\Delta S_{ab}}
\le \frac{\beta_n}{2}(n-1)\Delta_1+\beta_n^2\binom{n}{2}\Delta_0.
\]
Apply Theorem~\ref{thm:Delta0_bound} to bound $\Delta_0\le \frac{2n-1}{n}\Delta_1$.
Since $\beta_n=\frac{1}{2n-1}$, we obtain
\[
\beta_n^2\binom{n}{2}\Delta_0
\le \beta_n^2\binom{n}{2}\frac{2n-1}{n}\Delta_1
=\frac{n-1}{2(2n-1)}\Delta_1
=\frac{\beta_n}{2}(n-1)\Delta_1.
\]
Therefore $\sum_{a<b}\frac12\abs{\Delta S_{ab}}\le \beta_n(n-1)\Delta_1$.
Substituting into \eqref{eq:Delta2_split} gives
\[
\Delta_2\le \frac{\beta_n}{2}(n-1)\Delta_1+\beta_n(n-1)\Delta_1
=\frac32\beta_n(n-1)\Delta_1
=\frac{3(n-1)}{2(2n-1)}\Delta_1,
\]
as required.
\end{proof}

\begin{lemma}[$\ell_1$ control by the $1$-slice]\label{lem:l1_vs_Delta1}
For any $\pvec,\qvec\in[0,\lam_n]^n$,
\[
\normone{\pvec-\qvec}\le K(n)\,\Delta_1,
\qquad
K(n):=\frac{2n-1}{n\,(1-\lam_n)^{n-1}}.
\]
In particular, if $\Delta_1=0$ then $\pvec=\qvec$.
\end{lemma}

\begin{proof}
Fix $j\in[n]$ and consider $F(t):=P_{\{j\}}(\qvec+t(\pvec-\qvec))$ for $t\in[0,1]$.
Then $F(1)-F(0)=\delta_{\{j\}}$.
By the mean value theorem, there exists $t_j\in(0,1)$ such that
\[
\delta_{\{j\}}=F'(t_j)=\nabla P_{\{j\}}(\uvec^{(j)})\cdot\xvec,
\qquad
\uvec^{(j)}:=\qvec+t_j(\pvec-\qvec)\in[0,\lam_n]^n.
\]
Since $P_{\{j\}}(\yvec)=y_j\prod_{\ell\neq j}(1-y_\ell)$, we have
\[
\frac{\partial}{\partial y_j}P_{\{j\}}(\yvec)=\prod_{\ell\neq j}(1-y_\ell),
\qquad
\frac{\partial}{\partial y_m}P_{\{j\}}(\yvec)=-y_j\prod_{\ell\neq j,m}(1-y_\ell)\ \ (m\neq j).
\]
Therefore
\[
\delta_{\{j\}}
=\Big(\prod_{\ell\neq j}(1-u^{(j)}_\ell)\Big)x_j
-\sum_{m\neq j}u^{(j)}_j\Big(\prod_{\ell\neq j,m}(1-u^{(j)}_\ell)\Big)x_m.
\]
Taking absolute values and using $u^{(j)}_\ell\le \lam_n$ gives
\begin{align*}
\Big(\prod_{\ell\neq j}(1-u^{(j)}_\ell)\Big)\abs{x_j}
&\le \abs{\delta_{\{j\}}}+
\sum_{m\neq j}u^{(j)}_j\Big(\prod_{\ell\neq j,m}(1-u^{(j)}_\ell)\Big)\abs{x_m}\\
\abs{x_j}
&\le \frac{\abs{\delta_{\{j\}}}}{(1-\lam_n)^{n-1}}+
\sum_{m\neq j}\frac{u^{(j)}_j}{1-u^{(j)}_m}\abs{x_m}
\le \frac{\abs{\delta_{\{j\}}}}{(1-\lam_n)^{n-1}}+\beta_n\sum_{m\neq j}\abs{x_m}.
\end{align*}
Summing over $j$ and noting $\sum_j\abs{\delta_{\{j\}}}=\Delta_1$ gives
\[
\normone{\pvec-\qvec}
\le \frac{\Delta_1}{(1-\lam_n)^{n-1}}+\beta_n(n-1)\normone{\pvec-\qvec}.
\]
Since $1-\beta_n(n-1)=\frac{n}{2n-1}$, we can absorb to obtain
\[
\normone{\pvec-\qvec}\le \frac{1}{1-\beta_n(n-1)}\cdot\frac{\Delta_1}{(1-\lam_n)^{n-1}}
=\frac{2n-1}{n\,(1-\lam_n)^{n-1}}\,\Delta_1,
\]
as claimed.
If $\Delta_1=0$ then $\normone{\pvec-\qvec}=0$ and hence $\pvec=\qvec$.
\end{proof}

We will need a simple extremal bound for Poisson--binomial masses under our small-parameter restriction.

\begin{lemma}[Poisson--binomial extremum under small odds]\label{lem:pb_extremum}
Let $N\ge 1$ and let $r_1,\dots,r_N\in[0,1)$.
Write odds $a_i:=\frac{r_i}{1-r_i}$ and assume $\sum_{i=1}^N a_i\le 1$.
Let $X=\sum_{i=1}^N X_i$ with independent $X_i\sim\mathrm{Ber}(r_i)$.
Then for every $m\in\{1,2,\dots,N\}$, $\mathbb{P}[X=m]$ is nondecreasing in each coordinate $r_i$.
If moreover $0\le r_i\le \lam\le \frac{1}{N+1}$ for all $i$, then for every $m\in\{1,\dots,N\}$,
\[
\mathbb{P}[X=m]\le \binom{N}{m}\lam^{m}(1-\lam)^{N-m}.
\]
\end{lemma}

\begin{proof}
Write $f_m:=\mathbb{P}[X=m]$.
Using the generating function factorization
\[
\prod_{i=1}^N\big((1-r_i)+r_i z\big)
= \left(\prod_{i=1}^N(1-r_i)\right)\prod_{i=1}^N(1+a_i z),
\]
we have $f_m=f_0\,e_m(a_1,\dots,a_N)$, where $e_m$ is the $m$th elementary symmetric polynomial.
For $m\ge 1$,
\[
e_m(a)=\frac1m\sum_{i=1}^N a_i\,e_{m-1}(a_1,\dots,\widehat{a}_i,\dots,a_N)
\le \frac1m\left(\sum_{i=1}^N a_i\right)e_{m-1}(a),
\]
(where $\widehat{a}_i$ indicates that $a_i$ is omitted)
so
\[
\frac{f_m}{f_{m-1}}=\frac{e_m}{e_{m-1}}\le \frac{1}{m}\sum_{i=1}^N a_i\le 1.
\]
Thus $f_{m-1}\ge f_m$ for all $m\ge 1$.

Now fix $j\in[N]$ and write $X^{(-j)}:=\sum_{i\neq j}X_i$.
Differentiating the pmf with respect to $r_j$ gives the standard identity
\[
\frac{\partial}{\partial r_j}\,\mathbb{P}[X=m]
=\mathbb{P}[X^{(-j)}=m-1]-\mathbb{P}[X^{(-j)}=m].
\]
Since removing one coordinate can only decrease $\sum_i a_i$, the law of $X^{(-j)}$ still satisfies the hypothesis
$\sum_{i\neq j}\frac{r_i}{1-r_i}\le 1$.
Therefore $\mathbb{P}[X^{(-j)}=m-1]\ge \mathbb{P}[X^{(-j)}=m]$ for every $m\ge 1$, and the derivative is nonnegative.
This proves coordinatewise monotonicity of $\mathbb{P}[X=m]$ for $m\ge 1$.

For the final inequality, if $r_i\le \lam\le 1/(N+1)$ then
$\sum_i \frac{\lam}{1-\lam}=\frac{N\lam}{1-\lam}\le 1$, so the monotonicity applies at the endpoint vector $(\lam,\dots,\lam)$.
Thus $\mathbb{P}[X=m]$ is maximized when all $r_i=\lam$, in which case $X\sim\mathrm{Bin}(N,\lam)$ and
$\mathbb{P}[X=m]=\binom{N}{m}\lam^m(1-\lam)^{N-m}$.
\end{proof}

\begin{lemma}[Recursive bound for $\Delta_k$]\label{lem:recursive}
Fix $k\in\{2,3,\dots,n\}$ and let $\pvec,\qvec\in[0,\lam_n]^n$.
If $\Delta_1=0$ then $\pvec=\qvec$ and hence $\Delta_k=0$ for all $k$.
If $\Delta_1>0$, then
\[
\frac{\Delta_k}{\Delta_1}
\le
\frac{\beta_n(n-k+1)}{k}\cdot\frac{\Delta_{k-1}}{\Delta_1}
+
\binom{n-1}{k-1}\lam_n^{k-1}(1-\lam_n)^{n-k-1}\cdot\frac{K(n)}{k},
\]
where $K(n)$ is as in Lemma~\ref{lem:l1_vs_Delta1}.
\end{lemma}

\begin{proof}
Fix $S\subseteq[n]$ with $|S|=k$ and fix $i\in S$.
Using $P_S(\yvec)=\frac{y_i}{1-y_i}P_{S\setminus\{i\}}(\yvec)$ we write
\begin{align*}
\delta_S
&=
\frac{p_i}{1-p_i}P_{S\setminus\{i\}}(\pvec)-\frac{q_i}{1-q_i}P_{S\setminus\{i\}}(\qvec)\\
&=
\frac{p_i}{1-p_i}\big(P_{S\setminus\{i\}}(\pvec)-P_{S\setminus\{i\}}(\qvec)\big)
+\left(\frac{p_i}{1-p_i}-\frac{q_i}{1-q_i}\right)P_{S\setminus\{i\}}(\qvec)\\
&=\frac{p_i}{1-p_i}\,\delta_{S\setminus\{i\}}+\frac{p_i-q_i}{(1-p_i)(1-q_i)}\,P_{S\setminus\{i\}}(\qvec).
\end{align*}
Taking absolute values and summing over $i\in S$ gives
\[
k\abs{\delta_S}
\le
\sum_{i\in S}\left(\frac{p_i}{1-p_i}\abs{\delta_{S\setminus\{i\}}}
+\frac{\abs{x_i}}{(1-p_i)(1-q_i)}P_{S\setminus\{i\}}(\qvec)\right).
\]
Summing further over all $S$ with $|S|=k$ yields
\begin{align*}
k\Delta_k
&\le
\sum_{|S|=k}\sum_{i\in S}\frac{p_i}{1-p_i}\abs{\delta_{S\setminus\{i\}}}
+\sum_{|S|=k}\sum_{i\in S}\frac{\abs{x_i}}{(1-p_i)(1-q_i)}P_{S\setminus\{i\}}(\qvec).
\end{align*}

For the first term, $\frac{p_i}{1-p_i}\le \beta_n$ on $[0,\lam_n]$, and a multiplicity count gives
\[
\sum_{|S|=k}\sum_{i\in S}\abs{\delta_{S\setminus\{i\}}}
=(n-k+1)\sum_{|T|=k-1}\abs{\delta_T}
=(n-k+1)\Delta_{k-1}.
\]
Hence the first term is at most $\beta_n(n-k+1)\Delta_{k-1}$.

For the second term, re-index by $i$ and write $T=S\setminus\{i\}$ (so $|T|=k-1$ and $i\notin T$):
\begin{align*}
\sum_{|S|=k}\sum_{i\in S}\frac{\abs{x_i}}{(1-p_i)(1-q_i)}P_{S\setminus\{i\}}(\qvec)
&=
\sum_{i=1}^n\frac{\abs{x_i}}{(1-p_i)(1-q_i)}\sum_{\substack{T\subseteq[n]\setminus\{i\}\\ |T|=k-1}} P_T(\qvec).
\end{align*}
For each such $T$, the factor $(1-q_i)$ appears in $P_T(\qvec)$ (since $i\notin T$), so we can cancel it:
\[
\frac{1}{(1-p_i)(1-q_i)}P_T(\qvec)
=\frac{1}{1-p_i}\left(\prod_{j\in T}q_j\right)\left(\prod_{\ell\notin T,\ \ell\neq i}(1-q_\ell)\right).
\]
Thus
\[
\sum_{\substack{T\subseteq[n]\setminus\{i\}\\ |T|=k-1}}
\frac{1}{(1-p_i)(1-q_i)}P_T(\qvec)
=\frac{1}{1-p_i}\,\mathbb{P}\!\left[\sum_{\ell\neq i} Z_\ell = k-1\right],
\]
where $Z_\ell\sim\mathrm{Ber}(q_\ell)$ are independent.

Since $q_\ell\le \lam_n$ and $\lam_n\le 1/n=1/((n-1)+1)$, Lemma~\ref{lem:pb_extremum} (with $N=n-1$, $\lam=\lam_n$, and $m=k-1\ge 1$) gives
\[
\mathbb{P}\!\left[\sum_{\ell\neq i} Z_\ell = k-1\right]
\le \binom{n-1}{k-1}\lam_n^{k-1}(1-\lam_n)^{n-k}.
\]
Moreover $1-p_i\ge 1-\lam_n$, so $\frac{1}{1-p_i}\le \frac{1}{1-\lam_n}$.
Therefore the entire second term is bounded by
\[
\binom{n-1}{k-1}\lam_n^{k-1}(1-\lam_n)^{n-k-1}\sum_{i=1}^n\abs{x_i}.
\]
Apply Lemma~\ref{lem:l1_vs_Delta1} to bound $\sum_i\abs{x_i}\le K(n)\Delta_1$.
Putting everything together,
\[
k\Delta_k
\le \beta_n(n-k+1)\Delta_{k-1}
+\binom{n-1}{k-1}\lam_n^{k-1}(1-\lam_n)^{n-k-1}K(n)\,\Delta_1.
\]
Divide by $k\Delta_1$ (for $\Delta_1>0$) to obtain the claim.
\end{proof}

Define a numerical sequence $B_k(n)$ by $B_1(n):=1$ and, for $k\ge 2$,
\begin{equation}\label{eq:Bk_recurrence}
B_k(n)
:=
\frac{n-k+1}{k(2n-1)}\,B_{k-1}(n)
+
\binom{n-1}{k-1}\left(\frac{1}{2n-1}\right)^{k-1}\frac{2}{k}.
\end{equation}

\begin{corollary}[Universal slice bounds]\label{cor:Delta_k_bound}
For every $k\in\{2,\dots,n\}$ and every $\pvec,\qvec\in[0,\lam_n]^n$,
\[
\Delta_k(\pvec,\qvec)\le B_k(n)\,\Delta_1(\pvec,\qvec).
\]
\end{corollary}

\begin{proof}
If $\Delta_1=0$ then $\pvec=\qvec$ by Lemma~\ref{lem:l1_vs_Delta1}, hence $\Delta_k=0$.
If $\Delta_1>0$, Lemma~\ref{lem:recursive} yields the recursion \eqref{eq:Bk_recurrence} as an upper bound, starting from $B_1(n)=1$.
\end{proof}

\begin{lemma}[Closed form for $B_k(n)$]\label{lem:closed_form_Bk}
For every integer $k$ with $2\le k\le n$,
\[
B_k(n)
=
\frac{2k-1}{k(k-1)}\binom{n-2}{k-2}\,\frac{n-1}{(2n-1)^{k-1}}.
\]
\end{lemma}

\begin{proof}
By induction on $k$.
For $k=2$,
\[
B_2(n)=\frac{3}{2}\cdot\binom{n-2}{0}\cdot\frac{n-1}{2n-1}=\frac{3(n-1)}{2(2n-1)},
\]
matching Theorem~\ref{thm:Delta2_bound}.
Assume the formula holds for $k-1\ge 2$.
Using \eqref{eq:Bk_recurrence} and $\binom{n-1}{k-1}=\frac{n-1}{k-1}\binom{n-2}{k-2}$, we compute
\begin{align*}
B_k(n)
&=
\frac{n-k+1}{k(2n-1)}\cdot
\frac{2(k-1)-1}{(k-1)(k-2)}\binom{n-2}{k-3}\frac{n-1}{(2n-1)^{k-2}}
+\frac{2}{k}\binom{n-1}{k-1}\frac{1}{(2n-1)^{k-1}}\\
&=
\frac{n-1}{k(2n-1)^{k-1}}
\left[
\frac{(n-k+1)(2k-3)}{(k-1)(k-2)}\binom{n-2}{k-3}
+\frac{2}{k-1}\binom{n-2}{k-2}
\right].
\end{align*}
Using $\binom{n-2}{k-3}=\binom{n-2}{k-2}\cdot\frac{k-2}{n-k+1}$, the bracket becomes
\[
\binom{n-2}{k-2}\left[\frac{2k-3}{k-1}+\frac{2}{k-1}\right]
=\binom{n-2}{k-2}\cdot\frac{2k-1}{k-1}.
\]
Substituting yields the desired formula.
\end{proof}

\begin{theorem}[Summation identity]\label{thm:sum_Bk}
For every $n\ge 2$,
\[
\sum_{k=2}^n B_k(n)=\frac{n-1}{n}.
\]
\end{theorem}

\begin{proof}
Set $t:=\frac{1}{2n-1}$.
Using Lemma~\ref{lem:closed_form_Bk} and the change of variables $j=k-2$, we obtain
\begin{align*}
\sum_{k=2}^n B_k(n)
&=
\sum_{j=0}^{n-2}
\frac{2j+3}{(j+1)(j+2)}\binom{n-2}{j}(n-1)t^{j+1}\\
&=\frac{t}{n}
\sum_{k=2}^{n}\binom{n}{k}(2k-1)t^{k-2},
\end{align*}
where we used the identity
$\frac{1}{(j+1)(j+2)}\binom{n-2}{j}=\frac{1}{n(n-1)}\binom{n}{j+2}$.

Let $A(t):=\sum_{k=0}^n\binom{n}{k}t^k=(1+t)^n$.
Then $A'(t)=\sum_{k=1}^n k\binom{n}{k}t^{k-1}=n(1+t)^{n-1}$.
A short computation gives
\[
\sum_{k=0}^n\binom{n}{k}(2k-1)t^{k-2}
=\frac{2n}{t}(1+t)^{n-1}-\frac{1}{t^2}(1+t)^n.
\]
Subtracting the $k=0$ and $k=1$ terms (equal to $-t^{-2}$ and $nt^{-1}$) yields
\[
\sum_{k=2}^n\binom{n}{k}(2k-1)t^{k-2}
=\frac{2n}{t}(1+t)^{n-1}-\frac{1}{t^2}(1+t)^n+\frac{1}{t^2}-\frac{n}{t}.
\]
Multiplying by $t/n$ gives
\[
\sum_{k=2}^n B_k(n)=2(1+t)^{n-1}-1-\frac{(1+t)^n-1}{nt}.
\]
Finally, with $t=\frac{1}{2n-1}$ we have $1+t=\frac{2n}{2n-1}$ and
\[
\frac{(1+t)^n-1}{nt}
=\frac{2n-1}{n}\left(\left(\frac{2n}{2n-1}\right)^n-1\right)
=2\left(\frac{2n}{2n-1}\right)^{n-1}-\frac{2n-1}{n}.
\]
Substituting cancels the $(1+t)^{n-1}$ terms and yields
\[
\sum_{k=2}^n B_k(n)=-1+\frac{2n-1}{n}=\frac{n-1}{n}.
\]
\end{proof}

\begin{theorem}\label{thm:SMAL}
For every $\pvec,\qvec\in[0,\lam_n]^n$,
\[
\sum_{k=2}^n\Delta_k(\pvec,\qvec)\ \le\ \frac{n-1}{n}\,\Delta_1(\pvec,\qvec).
\]
\end{theorem}

\begin{proof}
If $\Delta_1=0$ then $\pvec=\qvec$ by Lemma~\ref{lem:l1_vs_Delta1}, so every $\Delta_k=0$.
Assume $\Delta_1>0$.
By Corollary~\ref{cor:Delta_k_bound},
$\Delta_k/\Delta_1\le B_k(n)$ for $k=2,\dots,n$.
Summing over $k$ and using Theorem~\ref{thm:sum_Bk} gives
\[
\sum_{k=2}^n\frac{\Delta_k}{\Delta_1}\le \sum_{k=2}^n B_k(n)=\frac{n-1}{n}.
\]
Multiplying by $\Delta_1$ completes the proof.
\end{proof}

Combining Theorem~\ref{thm:Delta0_bound} and Theorem~\ref{thm:SMAL} gives
\[
\sum_{k=0}^n\Delta_k
\le \left(1+\frac{2n-1}{n}+\frac{n-1}{n}\right)\Delta_1
=\left(4-\frac{2}{n}\right)\Delta_1,
\]
which immediately implies Theorem~\ref{thm:intro_small_tv}.

\appendix

\section{An $\ell_2$ bound for quasi-symmetric pairs}
\label{sec:appendix_quasisym}

It is shown in
\cite[Theorem 1.2]{kon25tens} that
for \emph{symmetric} pairs satisfying $\qvec=\onevec-\pvec$,
we have
$\TV(\Ber(\pvec),\Ber(\qvec))\le\norm{\pvec-\qvec}_2$.
Here we record a simple extension to a larger class of \emph{quasi-symmetric} pairs, with an additional factor $\sqrt{2}$.

\begin{definition}[Quasi-symmetric pairs]
A pair $(u,v)\in[0,1]^2$ is \emph{quasi-symmetric} if either $u\le \tfrac12\le v$ or $v\le \tfrac12\le u$.
Vectors $\pvec,\qvec\in[0,1]^n$ are quasi-symmetric if each coordinate pair $(p_i,q_i)$ is quasi-symmetric.
\end{definition}

\begin{theorem}\label{thm:sqrt2}
If $\pvec,\qvec\in[0,1]^n$ are quasi-symmetric, then
\[
\TV(\Ber(\pvec),\Ber(\qvec))\ \le\ \sqrt{2}\,\|\pvec-\qvec\|_2.
\]
\end{theorem}
\noindent
Remark: the choice
$\pvec=(1,1)$,
$\qvec=\bigl(\tfrac12,\tfrac12\bigr)$ shows that the optimal
constant must be at least $3/\sqrt{8}$.

\begin{lemma}[\cite{tsybakov09}, Disp. (2.20)]
\label{lem:tv-bc}
For any probability measures $P,Q$ on a finite set $\Omega$,
\[
\TV(P,Q) \le 
\sqrt{1-\left(\sum_{\omega\in\Omega}\sqrt{P(\omega)\,Q(\omega)}\right)^2}.
\]
\end{lemma}

\begin{lemma}[\cite{tsybakov09}, p. 83]
\label{lem:bc-tensor}
Let $P=\Ber(\pvec)$ and $Q=\Ber(\qvec)$ on $\{0,1\}^n$. Then
\[
\sum_{x\in\set{0,1}^n}\sqrt{P(x)Q(x)}
=\prod_{i=1}^n\left(\sqrt{p_iq_i}+\sqrt{(1-p_i)(1-q_i)}\right).
\]
\end{lemma}

\begin{lemma}
\label{lem:1d}
Let $p,q\in[0,1]$ satisfy $p\ge \tfrac12\ge q$ and define
\[
b(p,q):=\sqrt{pq}+\sqrt{(1-p)(1-q)}.
\]
Then
\[
1-b(p,q)^2\ \le\ 2(p-q)^2.
\]
\end{lemma}

\begin{proof}
A direct expansion shows the identity
\begin{equation}
\label{eq:identity}
1-b(p,q)^2=\left(\sqrt{p(1-q)}-\sqrt{q(1-p)}\right)^2.
\end{equation}
Let $A=p(1-q)$ and $B=q(1-p)$. Then $A-B=p-q$. Using
\[
\sqrt{A}-\sqrt{B}=\frac{A-B}{\sqrt{A}+\sqrt{B}},
\]
and \eqref{eq:identity}, we obtain
\begin{equation}
\label{eq:frac}
1-b(p,q)^2
=\frac{(p-q)^2}{\left(\sqrt{p(1-q)}+\sqrt{q(1-p)}\right)^2}.
\end{equation}

It remains to lower bound the denominator.
Define unit vectors in $\mathbb{R}^2$:
\[
\vec u=(\sqrt{p},\sqrt{1-p}),\qquad \vec v=(\sqrt{1-q},\sqrt{q}),
\qquad \|\vec u\|_2=\|\vec v\|_2=1.
\]
Then
\[
\vec u\cdot \vec v=\sqrt{p(1-q)}+\sqrt{q(1-p)}.
\]
Let $\theta$ and $\phi$ be the angles of $\vec u$ and $\vec v$ from the $x$-axis. Since $p\ge 1/2$,
we have $\theta\in[0,\pi/4]$, and since $q\le 1/2$, we have $\phi\in[0,\pi/4]$.
Therefore $|\theta-\phi|\le \pi/4$, so
\[
\vec u\cdot \vec v=\cos(\theta-\phi)\ge \cos(\pi/4)=\frac{1}{\sqrt{2}}.
\]
Substituting into \eqref{eq:frac} yields
\[
1-b(p,q)^2\le \frac{(p-q)^2}{(1/\sqrt2)^2}=2(p-q)^2,
\]
as desired.
\end{proof}

\begin{proof}[Proof of Theorem~\ref{thm:sqrt2}]
Since any pair $p_i,q_i$ may be simultaneously reflected about $1/2$ without affecting the TV,
we may assume, for all $i$,
\[
p_i\ge \tfrac12\ge q_i.
\]
By Lemmas~\ref{lem:tv-bc} and \ref{lem:bc-tensor},
\[
\TV(P,Q)\le \sqrt{1-\prod_{i=1}^n b_i^2},
\qquad
b_i:=\sqrt{p_iq_i}+\sqrt{(1-p_i)(1-q_i)}\in[0,1].
\]
For numbers $y_i\in[0,1]$ one has $1-\prod_i y_i \le \sum_i (1-y_i)$, so with $y_i=b_i^2$,
\[
\TV(P,Q)\le \sqrt{\sum_{i=1}^n (1-b_i^2)}.
\]
By Lemma~\ref{lem:1d} applied coordinatewise (recall $p_i\ge \tfrac12\ge q_i$), we have
\[
1-b_i^2\le 2(p_i-q_i)^2.
\]
Therefore
\[
\TV(P,Q)\le \sqrt{\sum_{i=1}^n 2(p_i-q_i)^2}
=\sqrt{2}\,\|\pvec-\qvec\|_2,
\]
which completes the proof.
\end{proof}

\paragraph{{\bf Acknowledgments}}
This research was supported in part by 
the Israel Science Foundation ISF grant
581/25
and the Binational Science Foundation BSF grant
2024243.


\begin{thebibliography}{99}

\bibitem{DBLP:conf/ijcai/0001GMMPV23}
Arnab Bhattacharyya, Sutanu Gayen, Kuldeep S. Meel, Dimitrios Myrisiotis, A. Pavan, and N. V. Vinodchandran.
\newblock On Approximating Total Variation Distance.
\newblock In \emph{Proceedings of the Thirty-Second International Joint Conference on Artificial Intelligence, {IJCAI} 2023, 19th-25th August 2023, Macao, SAR, China}, pages 3479--3487. ijcai.org, 2023.
\newblock doi: 10.24963/IJCAI.2023/387.

\bibitem{FengApproxTV23}
Weiming Feng, Heng Guo, Mark Jerrum, and Jiaheng Wang.
\newblock A simple polynomial-time approximation algorithm for the total variation distance between two product distributions.
\newblock In \emph{2023 Symposium on Simplicity in Algorithms (SOSA)}, pages 343--347, 2023.
\newblock doi: 10.1137/1.9781611977585.ch30.

\bibitem{Feng24Deterministically}
Weiming Feng, Liqiang Liu, and Tianren Liu.
\newblock On Deterministically Approximating Total Variation Distance.
\newblock In \emph{Proceedings of the 2024 Annual ACM-SIAM Symposium on Discrete Algorithms (SODA)}, pages 1766--1791, 2024.
\newblock doi: 10.1137/1.9781611977912.70.

\bibitem{bhattacharyya2024total}
Arnab Bhattacharyya, Sutanu Gayen, Kuldeep S. Meel, Dimitrios Myrisiotis, A. Pavan, and N. V. Vinodchandran.
\newblock Total Variation Distance Meets Probabilistic Inference.
\newblock In \emph{Forty-first International Conference on Machine Learning}, 2024.

\bibitem{kon25tens}
Aryeh Kontorovich.
\newblock {On the tensorization of the variational distance}.
\newblock \emph{Electronic Communications in Probability} 30: 1--10, 2025.
\newblock doi: 10.1214/25-ECP680.

\bibitem{kontorovich2026tvhomogenization}
Aryeh Kontorovich.
\newblock {TV} homogenization inequalities, preprint.
\newblock arXiv:2601.04079, math.PR, 2026.

\bibitem{KonAv26}
Aryeh Kontorovich and Ariel Avital.
\newblock Total variation over {B}ernoulli products: an {$O(\sqrt{\log n})$} approximation, in preparation.
\newblock 2026.

\bibitem{tsybakov09}
Alexandre B. Tsybakov.
\newblock \emph{Introduction to Nonparametric Estimation}.
\newblock Springer series in statistics. Springer, 2009.
\newblock doi: 10.1007/B13794.

\end{thebibliography}
\end{document}